\documentclass{article}

\usepackage{arxiv}

\usepackage[utf8]{inputenc} 
\usepackage[T1]{fontenc}    
\usepackage{hyperref}       
\usepackage{url}            
\usepackage{amsfonts}       
\usepackage{lipsum}	
\usepackage{amsmath}
\usepackage{amsthm}
\newtheorem{definition}{Definition}
\newtheorem{remark}{Remark}
\newtheorem{theorem}{Theorem}
\newtheorem{lemma}{Lemma}
\newtheorem{proposition}{Proposition}

 \numberwithin{equation}{section}

\title{Nonlinear Elliptic Equations With Variable Exponents Involving Singular Nonlinearity}

\author{
Hichem Khelifi$^{1}$ and Youssef El hadfi$^{2}$\\
   $^{1}$ Department
of Mathematics and Informatics, University of Algiers, Algiers, Algeria.\\
 2 Street Didouche Mourad Algiers.\\
 Applied Mathematics Laboratory, Badji Mokhtar University-Annaba B.P. 12, Algeria.\\
      \texttt{khelifi.hichemedp@gmail.com}\\
$^{2}$Laboratory LIPIM, National School of Applied Sciences Khouribga\\
Sultan Moulay Slimane University, Morocco\\
  \texttt{yelhadfi@gmail.com}\\}

\begin{document}
\maketitle
\begin{abstract}
 In this paper, we prove the existence and regularity of weak positive solutions for a class of nonlinear elliptic equations with a singular nonlinearity, lower order terms and $L^{1}$ datum in the setting of variable exponent Sobolev spaces. We will prove that the lower order term has some regularizing effects on the solutions. This work generalizes some
results given in \cite{1}.
\end{abstract}


\section{Introduction}
\subsection{Introduction of Our Problem}
This work is devoted to the study of the nonlinear elliptic problem
\begin{equation}
\label{2}
\begin{cases}
-\mbox{div}(a(x)\vert \nabla u\vert^{p(x)-2}\nabla u)+b(x)u\vert u\vert^{r(x)-1}=\frac{f}{ u^{\gamma(x)}} ,& \mbox{in}\;\; \Omega , \\
u>0, & \mbox{on}\;\; \Omega, \\
u=0, & \mbox{on}\;\; \partial\Omega,
\end{cases}
\end{equation}
where $\Omega$ is a bounded open subset of $\mathbb{R}^{N}$ ($N\geq2$) with Lipschitz boundary $\partial\Omega$, $f$ is a positive (that is $f(x)\geq0$ and not zero a.e.) function in $L^{1}(\Omega)$, and $p,r : \overline{\Omega}\rightarrow(0,+\infty)$, $\gamma : \overline{\Omega}\rightarrow(0,1)$ are continuous functions and satisfying
\begin{equation}
\label{3}
1<p^{-}:=\inf\limits_{x\in\overline{\Omega}}
p(x)\leq p^{+}:=\sup\limits_{x\in\overline{\Omega}}
p(x)<N,
\end{equation}
\begin{equation}
\label{4}
p(x)-1<r(x),
\end{equation}
\begin{equation}
\label{1}
0<\gamma^{-}:=\inf_{x\in \overline{\Omega}}\gamma(x)\leq
\gamma^{+}:=\sup_{x\in \overline{\Omega}}\gamma(x)<1,\quad   \quad \mbox{and} \quad \vert\nabla\gamma\vert\in L^{\infty}(\Omega)
\end{equation}
where $a(x), b(x)$ are measurable functions verifies for some
positive numbers $\alpha, \beta, \mu, \nu$ the following conditions
\begin{equation}
\label{5}
0<\alpha\leq a(x)\leq\beta,\quad 0<\mu\leq b(x)\leq\nu.
\end{equation}
\par Equations with variable exponents appear in various mathematical models. In some cases, they provide realistic models for the study of natural phenomena in electro-rheological fluids and important applications are related to image processing. We refer the reader to \cite{2,3,4} and the references therein.
\par For costant-exponent cases (i.e., $p(x)= p$, $r(x)= r$ and $\gamma(x)=\gamma$), the existence and regularity of solutions to problem \eqref{2} are studied in \cite{1,7', 14',14''}. They proved that the solution was in $W_{0}^{1,q}(\Omega)$ and $u^{r+\gamma}$ belongs to $L^{1}(\Omega)$, where $q=\frac{pr}{p+1-\gamma}$. The problem was also considered in \cite{5}, when $b(x)=0$ and $\gamma,p$ was a constants with $0\leq\gamma<1$, $f\in L^{m}(\Omega)$ ($m\geq1$). The authors in \cite{5} prove the existence and uniqueness results. If $p(x)=2$ and $\gamma,r$ were constants, the problem \eqref{2} has been treated in \cite{6}.
\par When the lower-order term does not appear in \eqref{2} (i.e., $b(x)=0$) and the exponent $p(x)\equiv p$, the problem \eqref{2}, have been treated in \cite{7}, under the hypothesis $f\in L^{m}(\Omega)$ ($m\geq1$). If $m=1$ and $0<\gamma^{-}\leq \gamma(x)\leq\gamma^{+}<1$ the authors proved that
the solution belongs to $W_{0}^{1,q}(\Omega)$, where $q=\frac{N(p+\gamma^{-}-1)}{N+\gamma^{-}-1}$.
\subsection{ Preliminary Work}
For some preliminary results on the Lebesgue and Sobolev spaces with variable exponent, we only give the definition of $L^{p(.)}(\Omega)$, for more details, see \cite{8,9} or monographs \cite{10,11}.  For an open $\Omega\subset\mathbb{R}^N$, let $p:\Omega \rightarrow[1,+\infty)$ be a measurable function such that
$$1<p^-=\mathrm{ess\,inf}\,p,\quad p^{+}=\mathrm{ess\,sup}\,p<+\infty.$$
We define the Lebesgue space with variable exponent $L^{p(\cdot)}(\Omega)$ to consist of all
measurable functions $u :\Omega\rightarrow\mathbb{R}$ for which the convex modular
$$\rho_{p(\cdot)}(u)=\int\limits_{\Omega}|u|^{p(x)}dx,$$
is finite. The expression
$$\|u\|_{p(\cdot)}:=\|u\|_{L^{p(\cdot)}(\Omega)}=\inf\left\{\lambda>0,\ \rho_{p(\cdot)}\left(\frac{u}{\lambda}\right)\leq 1\right\},$$
defines a norm in $L^{p(\cdot)}(\Omega)$, called the Luxemburg norm, and $\left(L^{p(\cdot)}(\Omega),\ \|u\|_{p(\cdot)}\right)$ is a uniformly convex Banach space. Its dual space is isomorphic to $L^{p'(\cdot)}(\Omega)$,
where $\frac{1}{p(x)}+\frac{1}{p'(x)}=1$. For all $u\in L^{p(\cdot)}(\Omega)$ and $v\in L^{p'(\cdot)}(\Omega)$, the H\"{o}lder type inequality
$$\left|\int\limits_{\Omega}uvdx\right|\leq\left(\frac{1}{p^{-}}+\frac{1}{p'^{-}}\right)\|u\|_{p(\cdot)}\|v\|_{p'(\cdot)}\leq2\|u\|_{p(\cdot)}\|v\|_{p'(\cdot)},$$
holds true. We define the Sobolev space with variable exponent
$$W^{1,p(\cdot)}(\Omega) = \left\{u \in  L^{p(\cdot)}(\Omega)\ \text{and} \ |\nabla u|\in L^{p(\cdot)}(\Omega)\right\},$$
which is equipped with the following norm
 $$\|u\|_{1,p(\cdot)}=\|u\|_{W^{1,p(\cdot)}(\Omega)}=\|u\|_{p(\cdot)}+\|\nabla u\|_{p(\cdot)}.$$
The space $\left(W^{1,p(\cdot)}(\Omega),\|u\|_{1,p(\cdot)}\right)$ is a reflexive Banach space.
Next, we define also
$$W_{0}^{1,p(\cdot)}(\Omega)=\left\{u \in  W^{1,p(\cdot)}(\Omega),\ u=0 \ \text{on}\ \partial\Omega \right\},$$
endowed with the norm $\|.\|_{1,p(\cdot)}$.\\
 The space $W_{0}^{1,p(.)}(\Omega)$ is separable and reflexive provided that with $1<p^{-}\leq p^{+}<\infty$.
\par An important role in manipulating the generalized Lebesgue and Sobolev spaces is played
by the modular $\varrho_{p(.)}(\Omega)$ of the space $L^{p(.)}(\Omega)$. We have the following result
\begin{proposition}[\cite{10}]\label{pro1}
 If $(u_n), u\in L^{p(\cdot)}(\Omega)$ and $p^{+} < +\infty$, then the
following properties hold true:
\begin{itemize}
  \item[(i)] $\min\left(\rho_{p(\cdot)}(u)^{\frac{1}{p^{+}}},\rho_{p(\cdot)}(u)^{\frac{1}{p^{-}}}\right)\leq\|u\|_{p(\cdot)}
      \leq\max\left(\rho_{p(\cdot)}(u)^{\frac{1}{p^{+}}},\rho_{p(\cdot)}(u)^{\frac{1}{p^{-}}}\right),$
  \item[(ii)] $\min\left(\|u\|_{p(\cdot)}^{p^{-}},\|u\|_{p(\cdot)}^{p^{+}}\right)\leq\rho_{p(\cdot)}(u)
      \leq\max\left(\|u\|_{p(\cdot)}^{p^{-}},\|u\|_{p(\cdot)}^{p^{+}}\right),$
  \item[(iii)]$ \|u\|_{p(\cdot)}\leq\rho_{p(\cdot)}(u)+1,$
\end{itemize}
\end{proposition}
\par Next, we recall some embedding results regarding variable exponent
Lebesgue-Sobolev spaces. If $p,\theta : \Omega\rightarrow (1,+\infty)$ are Lipschitz continuous function satisfying \eqref{3} and $p(x)\leq \theta(x)\leq p^{*}(x)$ for any $x\in \Omega$, where $p^{*}(x)=\frac{Np(x)}{N-p(x)}$, then there exists a compact embedding
\begin{equation}
\label{03}
W^{1,p(.)}(\Omega)\hookrightarrow\hookrightarrow L^{\theta(.)}(\Omega)\hookrightarrow L^{\theta^{-}}(\Omega),
\end{equation}
where $\theta^{-}=\inf\limits_{x\in \overline{\Omega}}\theta(x)$.
\subsection{Statement of main result}
\label{S.results}
\par We use the following definition of the weak solutions.
\begin{definition}
\label{defn1}
\it  Let $f\in L^{1}(\Omega)$. A  function $u\in W_{0}^{1,1}(\Omega)$ is a weak solution to problem \eqref{2}, if
\begin{align*}
&\forall \omega\subset\subset\Omega,\;\; \exists c_{\omega}>0\quad\mbox{such that}\quad u\geq c_{\omega}\;\; \mbox{a.e. in}\;\; \omega,\quad u^{r(x)}\in L^{1}(\Omega),
\end{align*}
and
\begin{equation}
\label{6}
\int_{\Omega}a(x)\vert \nabla u\vert^{p(x)-2}\nabla u.\nabla\varphi dx+\int_{\Omega}b(x)
 u^{r(x)}\varphi dx=\int_{\Omega} \frac{f\varphi}{u^{\gamma(x)}}dx,
\end{equation}
for every $\varphi\in C^1_0(\Omega)$.
 \end{definition}
In this paper we will show the following result.
\begin{theorem}
\label{thm0}
\it Suppose that assumptions \eqref{3}-\eqref{1} hold. Let $f\in L^{1}(\Omega)$ be a positive function ($f(x)\geq0$ and
not zero a.e.). Assume that
\begin{equation}
\label{7}
 p(x)>1+\frac{1-\gamma(x)}{r(x)}.
\end{equation}
Then, the problem \eqref{2} has at least one weak solution $u\in W_{0}^{1,q(.)}(\Omega)$, with
\begin{equation}
\label{8}
q(x)=\frac{p(x)}{1+\frac{1-\gamma(x)}{r(x)}}.
\end{equation}
Moreover $u^{r(x)+\gamma(x)}$ belongs to $L^{1}(\Omega)$.
\end{theorem}
\begin{remark}
\label{rem3}\qquad\quad
\begin{description}
  \item[$\bullet$] The assumption \eqref{1} implies $1<q(.)<p(.).$
 \item[$\bullet$]The assumption \eqref{4} implies $q(.)>p(.)-1.$
\end{description}
\end{remark}
\par In order to prove this result, we will work by approximation, "truncating" the singular term $\frac{1}{u^{\gamma(x)}}$ so that it becomes not singular at the origin. We will get some a priori estimates on the solutions $u_{n}$ of the
approximating problems, which will allow us to pass to the limit and
find a solution to problem \eqref{2}.
\section{Approximating problems}
\par Hereafter, we denote by $T_{k}$ the truncation function at the level $k>0$, defined by $T_{k}(s)=\max\{-k,\min\{s,k\}\}$ for every $s\in\mathbb{R}$.
\par Let $(f_{n})$ ($f_{n}>0$) be a sequence of bounded functions defined in $\Omega$ which converges to $f>0$ in $L^{1}(\Omega)$, and which verifies the inequalities $f_{n}\leq n$ and $f_{n}\leq f$ for every $n\geq1$ (for example $f_{n}=T_{n}(f)$). Consider the following approximate equation
\begin{align}
\label{09}
\begin{cases}
-\mbox{div}(a(x)\vert \nabla u_{n}\vert^{p(x)-2}\nabla u_{n})+b(x)u_{n}\vert u_{n}\vert^{r(x)-1}
=\frac{f_{n}}{\left(  u_{n}+\frac{1}{n}\right)^{\gamma(x)}} ,& \mbox{in}\;\; \Omega , \\
u_{n}=0, & \mbox{on}\;\; \partial\Omega.
\end{cases}
\end{align}
\begin{theorem}
\label{thm1}
\it Let $f\in L^{1}(\Omega)$, and let $r,p : \overline{\Omega}\rightarrow(1,+\infty)$, $\gamma : \overline{\Omega}\rightarrow(0,1)$ are continuous functions. Assume that \eqref{3} and \eqref{5} holds true. Then the problem \eqref{09} has a nonnegative solution $u_{n}\in W_{0}^{1,p(.)}(\Omega)$.
\end{theorem}
\begin{lemma}\cite{14}
\label{lem1}
\it Suppose that the hypotheses of Theorem \ref{thm1} are satisfied. Then there exists at least one solution $u_{n}\in W_{0}^{1,p(.)}(\Omega)\cap L^{\infty}(\Omega)$ to the problem \eqref{09} in the sense that
\begin{align}
\label{10}
&\int_{\Omega}a(x)\vert \nabla u_{n}\vert^{p(x)-2}\nabla u_{n}.\nabla\varphi+\int_{\Omega}b(x)u_{n}
\vert u_{n}\vert^{r(x)-1}\varphi=\int_{\Omega}\frac{f_{n}}{\left(  u_{n}+\frac{1}{n}\right)^{\gamma(x)}}\varphi,
\end{align}
for every $\varphi\in W_{0}^{1,p(.)}(\Omega)\cap L^{\infty}(\Omega)$.
\end{lemma}

\begin{proof}
This proofs derived from Schauder-Tychonov fixed point Theorem (see, for example,  \cite[p. 581]{13'}, \cite[p. 298]{13''}).
 Let $n$ in $\mathbb{N}$ be fixed, let $v$ be a function in $L^{p(.)}(\Omega),$ we know that the following non singular problem
\begin{equation}\label{10'}
\left\{\begin{array}{cl}
-\operatorname{div}\left(a(x)|\nabla w|^{p(x)-2} \nabla w\right)+b(x)|w|^{r(x)-1} w=\frac{f_{n}}{\left(|v|+\frac{1}{n}\right)^{\gamma(x)}} & \text { in } \Omega \\
w=0 & \text { on } \partial \Omega .
\end{array}\right.
\end{equation}
Therefore, the Minty-Browder Theorem (see, e.g. \cite{5'}) implies that problem \eqref{10'} has a unique solution $w \in W_{0}^{1, p(x)}(\Omega).$
Let us define a map
$$
G: L^{p(.)}(\Omega) \rightarrow L^{p(.)}(\Omega)
$$
and define $w=G(v)$ to be the unique solution of \eqref{10'}. Taking $w$ as test function, we have
$$\alpha \int_{\Omega}|\nabla w|^{p(x)} \leq \int_{\Omega} a(x)|\nabla w|^{p(x)-2} \nabla w \cdot \nabla w=\int_{\Omega} \frac{f_{n} w}{\left(|v|+\frac{1}{n}\right)^{\gamma(x)}} \leq n^{\gamma_{+}+1} \int_{\Omega}|w| .$$
By the Sobolev inequality on the left hand side and the H\"older inequality on the right hand side one has
$$
\left[\int_{\Omega}|w|^{p^{*}(x)}\right]^{p(x) / p^{*}(x)} \leq C n^{\gamma_{+}+1}\left(\int_{\Omega}|w|^{p^{*}(x)}\right)^{\frac{1}{p^{*}(x)}}
$$
for some constant $C$ independent on $v.$ This implies
$$
\|w\|_{L^{p^{*}(.)}(\Omega)} \leq (C n^{\gamma_{+}+1})^{\frac{1}{p(x)-1}}=C_{n}
$$
so that the ball of $L^{p^{*}(.)}(\Omega)$ of radius $C_{n}$ is invariant for $G.$ It is easy to prove, using the Sobolev embedding, that $G$ is both continuous and compact on $L^{p^{*}(.)}(\Omega),$ so that by Schauder's fixed point Theorem there exists $u_{n}$ in $W_{0}^{1, p(x)}(\Omega),$ for every fixed $n$ such that $u_{n}=S\left(u_{n}\right),$ i.e., $u_{n}$ solves
\begin{equation}\label{Pn}
\left\{
\begin{array}{cl}
-\operatorname{div}\left(a(x)\left|\nabla u_{n}\right|^{p(x)-2} \nabla u_{n}\right)+b(x)\left|u_{n}\right|^{r(x)-1} u_{n}=\frac{f_{n}}{\left(\left|u_{n}\right|+\frac{1}{n}\right)^{\gamma(x)}} & \text { in } \Omega \\
u_{n}=0 & \text { on } \partial \Omega
\end{array}
\right.
\end{equation}
Using as a test function $u_{n}^{-}=\min{\{u_{n},0\}}$ , one has $u_{n}\geq 0,$ Since the right hand
side of \eqref{09} belongs to $L^{\infty}(\Omega)$ and we proceed in the same way as \cite{stamp} we obtain $u_{n}$ belongs to $L^{\infty}(\Omega).$  (although its norm in $L^{\infty}(\Omega)$ may depend on $n$).
\end{proof}
\begin{lemma}
\label{lem2}
Suppose that the hypotheses of Theorem \ref{thm1} are satisfied. Then the sequence $u_{n}$ is increasing with respect to $n$, $u_{n}>0$ in $\Omega$, and for every $\omega\subset\subset \Omega$ there exists $c_{\omega}>0$ (independent on $n$) such that
\begin{equation}
\label{20}
u_{n}(x)\geq c_{\omega}>0, \quad \forall x\in\Omega,\quad \forall n\in\mathbb{N}.
\end{equation}
Moreover there exists the pointwise limit $u\geq c_{\omega}$ of the sequence $u_{n}$.
\end{lemma}
\begin{proof}[Proof of the Lemma \ref{lem2}]
Due to $0\leq f_{n}\leq f_{n+1}$ and $\gamma(x)>0$, we have that
\begin{align*}
-\mbox{div}(a(x)\vert\nabla u_{n}\vert^{p(x)-2}\nabla u_{n})+b(x)u_{n}^{r(x)}=\frac{f_{n}}{\left(u_{n}+\frac{1}{n}\right)^{\gamma(x)}}
\quad\leq \frac{f_{n+1}}{\left(u_{n}+\frac{1}{n+1}\right)^{\gamma(x)}}.
\end{align*}
So that
\begin{align}
\label{21}
&-\mbox{div}(a(x)\vert\nabla u_{n}\vert^{p(x)-2}\nabla u_{n})+\mbox{div}(a(x)\vert\nabla u_{n+1}\vert^{p(x)-2}\nabla u_{n+1}) +b(x)u_{n}^{r(x)}-b(x)u_{n+1}^{r(x)}\nonumber\\
&\qquad\leq f_{n+1}\Bigg[\frac{\left(u_{n+1}+\frac{1}{n+1}\right)^{\gamma(x)}
-\left(u_{n}+\frac{1}{n+1}\right)^{\gamma(x)}}{\left(u_{n}+\frac{1}{n+1}\right)^{\gamma(x)}
\left(u_{n+1}+\frac{1}{n+1}\right)^{\gamma(x)}}\Bigg].
\end{align}
We now choose $(u_{n}-u_{n+1})_{+}=\max\{u_{n}-u_{n+1},0\}$ as test function in \eqref{21}. In the left hand side we use \eqref{5} and the monotonicity of the $p(x)-$laplacian operator as well as the monotonicity of the function $t\rightarrow \vert t\vert^{r(x)-1}t$. For the right hand, using the fact that $\gamma(x)\geq0$ and $f_{n+1}\geq0$, we have
\begin{align}
\label{22}
\Bigg[\left(u_{n+1}+\frac{1}{n+1}\right)^{\gamma(x)}
-\left(u_{n}+\frac{1}{n+1}\right)^{\gamma(x)}\Bigg](u_{n}-u_{n+1})_{+}\leq0.
\end{align}
By \eqref{22}, we get
$$\alpha \int_{\Omega}\vert \nabla(u_{n}-u_{n+1})_{+}\vert^{p(x)}\leq0,$$
which implies that $(u_{n}-u_{n+1})_{+}=0$ a.e. in $\Omega$, that is, $u_{n}\leq u_{n+1}$ for every $n\in\mathbb{N}$. Since the sequence $(u_{n})$ is increasing with respect to $n$, we only need to prove that  \eqref{20} holds for $u_{1}$. Using Lemma \ref{lem1}, we know that $u_{1}\in L^{\infty}(\Omega)$, i.e., there exists a constant $c_{0}$ (depending only on $\Omega$ and $N$) such that $\Vert u_{1}\Vert_{L^{\infty}(\Omega)}\leq c\Vert f_{1}\Vert_{L^{\infty}}(\Omega)\leq c_{0}$, then
\begin{align*}
&-\mbox{div}(a(x)\vert\nabla u_{1}\vert^{p(x)-2}\nabla u_{1})+b(x)u_{1}^{r(x)}=\frac{f_{1}}{\left(u_{1}+1\right)^{\gamma(x)}}\geq \frac{f_{1}}{(c_{0}+1)^{\gamma(x)}}\geq0.
\end{align*}
Since $\frac{f_{1}}{(c_{0}+1)^{\gamma(x)}}$ is not identically zero, the strong maximum principle implies that $u_{1}>0$ in $\Omega$ (see \cite{15}). Since $u_{n}\geq u_{1}$ for every $n\in\mathbb{N}$, \eqref{20} holds for $u_{n}$ (with the same constant $c_{\omega}$ which is then independent on $n$).
\end{proof}
\begin{proof}[Proof of the Theorem \ref{thm1}]
In virtue of the Lemma \ref{lem1} and Lemma \ref{lem2}, there exists at least one nonnegative weak solution $u_{n}\in W_{0}^{1,p(.)}(\Omega)\cap L^{\infty}(\Omega)$ of problem \eqref{09}.
\end{proof}
\section{A priori estimates}
In the remainder of this section, we denote by $C_{i}$  $i=1,2,3,...$ various positive constants
depending only on the data of the problem, but not on $n$.
\begin{lemma}
\label{lem3}
Let $k>0$ be fixed. The sequence $(T_{k}(u_{n}))$, where $u_{n}$ is
a solution to \eqref{Pn}, is bounded in $W_{0}^{1,p(.)}(\Omega)$.
\end{lemma}
\begin{proof}
Taking $T_{k}(u_{n})$ as a test function in \eqref{Pn}, we obtain
\begin{align*}
&\int_{\Omega}a(x)\vert \nabla u_{n}\vert^{p(x)-2}\nabla u_{n}.\nabla T_{k}(u_{n})+\int_{\Omega}b(x)
u_{n}^{r(x)}T_{k}(u_{n})=\int_{\Omega}\frac{f_{n}}{\left( \vert u_{n}\vert+\frac{1}{n}\right)^{\gamma(x)}}T_{k}(u_{n}).
\end{align*}
Using \eqref{5}, $f_{n}\leq f$, $T_{k}(u_{n})\neq0$, and dropping the nonegative order term, we get
\begin{equation}
\label{23}
\int_{\Omega}\vert \nabla T_{k}(u_{n})\vert^{p(x)}dx\leq \frac{k}{\alpha}\Vert f\Vert_{L^{1}(\Omega)}.
\end{equation}
As a consequence of Proposition \ref{pro1} and \eqref{23}, $T_{k}(u_{n})$ is bounded in $W_{0}^{1,p(.)}(\Omega)$
\end{proof}
\begin{lemma}
\label{lem4}
Suppose that the hypotheses of Theorem \ref{thm0} are satisfied. Then, the sequence $u_{n}$ is bounded in $ W_{0}^{1,q(.)}(\Omega)$, where
$q(.)$ is given by \eqref{8}.
Moreover $(u_{n}^{r(x)+\gamma(x)})$ belongs to $L^{1}(\Omega)$.
\end{lemma}
\begin{proof}
Taking
$\varphi(x,u)=(u_{n}+1)^{\gamma (x)}-1$, as test function in \eqref{Pn}, by \eqref{1}, \eqref{5}, and the fact that for a.e. $x\in \Omega$
$$\nabla \varphi(x,u)=\nabla\gamma(x)(u_{n}+1)^{\gamma(x)}\ln(u_{n}+1)+
\gamma(x)\frac{\nabla u_{n}}{(u_{n}+1)^{\gamma(x)}},$$
we obtain
\begin{align*}
&\gamma^{-}\alpha\int_{\Omega}
\frac{\vert \nabla u_{n}\vert^{p(x)}}{(1+u_{n})^{1-\gamma(x)}}+
\mu\int_{\Omega}
u_{n}^{r(x)}\big[(u_{n}+1)^{\gamma(x)}-1
\big]\nonumber\\
&\qquad\leq C_{1}\int_{\Omega} \vert \nabla u_{n}\vert^{p(x)-1}(u_{n}+1)^{\gamma(x)}
\ln(u_{n}+1)+
\int_{\Omega}
f\frac{\big[(u_{n}+1)^{\gamma(x)}-1
\big]}{\left(u_{n}+\frac{1}{n}\right)^{\gamma(x)}}.
\end{align*}
Using the fact that $\vert u_{n}\vert^{\theta(x)}\geq 2^{1-\theta^{+}}(1+u_{n})^{\theta(x)}-1$ (here $\theta(x)=r(x)$ and $\theta(x)=\gamma(x)$), we have
 \begin{align}
 \label{24}
&\gamma^{-}\alpha\int_{\Omega}
\frac{\vert \nabla u_{n}\vert^{p(x)}}{(1+u_{n})^{1-\gamma(x)}}+
2^{1-r^{+}}\mu\int_{\Omega}
(u_{n}+1)^{r(x)+\gamma(x)}\nonumber\\
&\qquad\leq C_{2}+
\frac{1}{2^{1-\gamma^{+}}}\int_{\Omega}f+ C_{1}\int_{\Omega} \vert \nabla u_{n}\vert^{p(x)-1}(u_{n}+1)^{\gamma(x)}
\ln(u_{n}+1).
\end{align}
We can estimate the last term in \eqref{24} by application of Young's inequality
 \begin{align}
 \label{25}
&(1+u_{n})^{\gamma(x)}\ln(1+u_{n})
\vert u_{n}\vert^{p(x)-1}\nonumber\\
&\qquad= (1+u_{n})^{1-\frac{1-\gamma(x)}{p(x)}}\ln(1+u_{n})
\vert u_{n}\vert^{p(x)-1}(1+u_{n})^{-\frac{(1-\gamma(x))(p(x)-1)}{p(x)}}\nonumber\\
&\qquad\leq C_{3}(1+u_{n})^{p(x)-(1-\gamma(x))}
(\ln(1+u_{n}))^{p(x)}
+\varepsilon\frac{\vert \nabla u_{n}\vert^{p(x)}}{(u_{n}+1)^{1-\gamma(x)}}.
\end{align}
We choose $\varepsilon=\frac{\gamma^{-}\alpha}{2C_{1}}$. By \eqref{24} and \eqref{25} we obtain
\begin{align}
 \label{26}
&\frac{1}{2}\int_{\Omega}
\frac{\vert \nabla u_{n}\vert^{p(x)}}{(1+u_{n})^{1-\gamma(x)}}+
2^{1-r^{+}}\mu\int_{\Omega}
(u_{n}+1)^{r(x)+\gamma(x)}\nonumber\\
&\qquad\leq C_{4}+ C_{5}\int_{\Omega} (u_{n}+1)^{p(x)-(1-\gamma(x))}
(\ln(u_{n}+1))^{p(x)}.
\end{align}
The hypothesis \eqref{4} implies  $(1+t)^{p(x)-1-r(x)-c}(\ln (1+t))^{p(x)}$ is bounded for all $x\in \overline{\Omega}$ and $t\in\mathbb{R}^{+}$. By another application of Youngs inequality, we have
\begin{align}
 \label{27}
&(u_{n}+1)^{p(x)-(1-\gamma(x))}(\ln(u_{n}+1))^{p(x)}\nonumber\\
&\quad=(u_{n}+1)^{r(x)+\gamma(x)+c}(u_{n}+1)^{p(x)-1-r(x)-c}(\ln(u_{n}+1))^{p(x)}\nonumber\\
&\quad\leq \varepsilon (u_{n}+1)^{r(x)+\gamma(x)}+C_{6}.
\end{align}
Therefore, by \eqref{26}, \eqref{27}, we get
\begin{align}
 \label{28}
\int_{\Omega}
\frac{\vert \nabla u_{n}\vert^{p(x)}}{(1+u_{n})^{1-\gamma(x)}}+
\int_{\Omega}
(u_{n}+1)^{r(x)+\gamma(x)}
&\leq C_{7}.
\end{align}
Since $r(x)\geq0$ and $\gamma(x)\geq0$, we have
\begin{align}
 \label{29}
\int_{\Omega}u_{n}^{r(x)}\leq\int_{\Omega}
(u_{n}+1)^{r(x)}\leq\int_{\Omega}
(u_{n}+1)^{r(x)+\gamma(x)}
&\leq C_{7}.
\end{align}
The inequality \eqref{29} implies that $(u_{n}^{r(x)+\gamma(x)})$ is bounded in $L^{1}(\Omega)$. Let $q(x)< p(x)$, using Young's inequality and \eqref{28}, we have
\begin{align}
 \label{30}
\int_{\Omega}\vert \nabla u_{n}\vert^{q(x)}&=\int_{\Omega}
\frac{\vert \nabla u_{n}\vert^{q(x)}}{(u_{n}+1)^{(1-\gamma(x))\frac{q(x)}{p(x)}}}\nonumber\\
&\leq C_{8}\int_{\Omega}\frac{\vert \nabla u_{n}\vert^{p(x)}}{(u_{n}+1)^{1-\gamma(x)}}+C_{9}
\int_{\Omega}(u_{n}+1)^{(1-\gamma(x))\frac{q(x)}{p(x)-q(x)}}\nonumber\\
&\leq C_{10}+C_{9}
\int_{\Omega}(u_{n}+1)^{(1-\gamma(x))\frac{q(x)}{p(x)-q(x)}}.
\end{align}
Set
\begin{equation*}
(1-\gamma(x))\frac{q(x)}{p(x)-q(x)}=r(x).
\end{equation*}
Then this equality and \eqref{29}-\eqref{30} yield
\begin{equation}
\label{030}
\int_{\Omega}\vert \nabla u_{n}\vert^{q(x)}\leq C_{11}.
\end{equation}
\end{proof}
\begin{lemma}
\label{lem5}
Let $u_{n}$ be a solution to problem \eqref{Pn}. Then
$$\int_{\{u_{n}>k\}}u_{n}^{r(x)}\leq \frac{1}{\mu k^{\gamma^{+}}}\int_{\{u_{n}>k\}}f,\quad \forall k>0,$$
$$\lim\limits_{\vert E\vert\rightarrow0}\int_{E}u_{n}^{r(x)}=0, \quad \mbox{uniformly with respect to}\; n,$$
for every measurable subset $E$ in $\Omega$.
\end{lemma}
\begin{proof}
Let $k>0$ and $\psi_{j}$ be a sequence of increasing, positive, uniformly bounded $C^{\infty}(\Omega)$ functions, such that
$\psi_{j}(s)\rightarrow \chi_{\{s>k\}}$, as $j\rightarrow +\infty$. Choosing $\psi_{j}(u_{n})$ in \eqref{Pn}, using \eqref{5}, we get
$$\mu\int_{\Omega}u_{n}^{r(x)}\psi_{j}(u_{n})\leq \int_{\Omega}\frac{f_{n}}{\left( u_{n}+\frac{1}{n}\right)^{\gamma(x)}}\psi_{j}(u_{n}).$$
Therefore, as $j$ tends to infinity and that $k^{\gamma^{-}}\leq\left(k+\frac{1}{n}\right)^{\gamma^{-}}\leq \left(u_{n}+\frac{1}{n}\right)^{\gamma(x)}$ in the set $\{u_{n}>k\}$, we obtain
\begin{equation}
\label{31}
\int_{\{u_{n}>k\}}u_{n}^{r(x)}\leq \frac{1}{\mu k^{\gamma^{-}}}\int_{\{u_{n}>k\}}f.
\end{equation}
By \eqref{31}, for any measurable subset $E$ in $\Omega$, we have
\begin{align}
 \label{32}
\int_{E}u_{n}^{r(x)}&=
\int_{E\cap\{u_{n}\leq k\}}u_{n}^{r(x)}
+\int_{E\cap\{u_{n}> k\}}u_{n}^{r(x)}
\leq k^{r^{+}}\vert E\vert+\frac{1}{\mu k^{\gamma^{-}}}\int_{\{u_{n}>k\}}f.
\end{align}
Since $f\in L^{1}(\Omega)$, we may choose $k=k_{\varepsilon}$ large enough such that
\begin{equation}
\label{33}
\int_{\{u_{n}>k\}}f\leq \varepsilon.
\end{equation}
Therefore, the estimates \eqref{32}-\eqref{33} implies that
$$\int_{E}u_{n}^{r(x)}\leq k_{\varepsilon}^{r^{+}}\vert E\vert+\frac{\varepsilon}{\mu k_{\varepsilon}^{\gamma^{-}}},$$
and the statement of this lemma is thus proved.
\end{proof}
\begin{lemma}
\label{lem6}
Let $u_{n}$ be a solution to problem \eqref{Pn}. Then
\begin{equation}
\label{033}
\lim\limits_{\vert E\vert\rightarrow0}\int_{E}
\vert\nabla u_{n}\vert^{q(x)}=0, \quad \mbox{uniformly with respect to}\; n,
\end{equation}
for every measurable subset $E$ in $\Omega$ and $q(.)$ given by \eqref{8}.
\end{lemma}
\begin{proof}
Let $\varepsilon>$, by Lemma \ref{lem4}, we may choose $k=k_{\varepsilon}$ large enough such that
\begin{equation}
\label{34}
\int_{E\cap\{u_{n}> k\}}\vert\nabla u_{n}\vert^{q(x)}\leq\varepsilon.
\end{equation}
From the estimate \eqref{23} and that $q(x)<p(x)$, we deduce
\begin{equation}
\label{35}
\int_{E\cap\{u_{n}\leq k\}}\vert\nabla T_{k}(u_{n})\vert^{q(x)}\leq\varepsilon.
\end{equation}
By \eqref{34} and \eqref{35}, for any measurable subset $E$ in $\Omega$, we have
\begin{align*}
\int_{E}\vert\nabla u_{n}\vert^{q(x)}&=
\int_{E\cap\{u_{n}\leq k\}}\vert\nabla u_{n}\vert^{q(x)}
+\int_{E\cap\{u_{n}> k\}}\vert\nabla u_{n}\vert^{q(x)}\leq 2\varepsilon.
\end{align*}
Therefore, we deduce that $\vert\nabla u_{n}\vert^{q(x)}$ is equiintegrable in $L^{1}(\Omega)$. Thus \eqref{033} is proved.
\end{proof}
\section{Proof of the main theorem}
\par By Lemma \ref{lem3}, the sequence $(u_{n})_{n}$ is bounded in $W_{0}^{1,q(.)}(\Omega)$. Therefore, there exists a function $u\in W_{0}^{1,q(.)}(\Omega)$ such that (up to a subsequence)
\begin{equation}
   \label{00}
  \left\{\begin{array}{l}
       u_{n}\rightharpoonup u\quad
      \text{in}\; W_{0}^{1,q(.)}(\Omega),\\
      u_{n}\rightarrow u\quad \text{a.e. in} \;  \Omega.
   \end{array}\right.
\end{equation}
\begin{proposition}
\label{prop0}
If the sequence $T_{k}(u_{n})$ of the truncates of the solutions $u_{n}$ of \eqref{Pn} is bounded in $W_{0}^{1,p(.)}(\Omega)$.
Then
\begin{equation}
\label{36}
 T_{k}(u_{n})\rightarrow  T_{k}(u)\quad \mbox{strongly in }\; W_{0}^{1,p(.)}(\Omega),
\end{equation}
as $n\rightarrow\infty$, for every $k>0$. In particular $\nabla u_{n}\rightarrow \nabla u$ a.e. in $\Omega$.
\end{proposition}
\begin{proof}
By Lemma \ref{lem3} $T_{k}(u_{n})$ is bounded in $W_{loc}^{1,p(.)}(\Omega)$, it weakly converges in this space to its pointwise limit $T_{k}(u)$. Moreover, since $f_{n}\geq0$ and $u_{n}\geq0$ a.e., we have that
\begin{equation*}
\label{37}
-\mbox{div}(a(x)\vert \nabla u_{n}\vert^{p(x)-2}\nabla u_{n})+b(x)u_{n}^{r(x)}\geq0,
\end{equation*}
for all $n\in \mathbb{N}$ and $k>0$.
\par Now we fix $\phi\in C_{0}^{1}(\Omega)$ such that $0\leq \phi\leq1$ on $\Omega$ and such that $\phi\equiv1$ on a fixed subset $\omega$ of $\Omega$. Then, thanks to the monotonicity of the $p(x)-$laplacian operator, \eqref{5}, and that
$T_{k}(u_{n})\geq T_{k}(u)$ (since $u_{n}\rightarrow u\leq u_{n}$), we can conclude that the following hold
\begin{align}
\label{38}
&0<\beta\int_{\omega}\left(\vert \nabla T_{k}(u_{n})\vert^{p(x)-2}\nabla T_{k}(u_{n})-\vert \nabla T_{k}(u)\vert^{p(x)-2}\nabla T_{k}(u)\right).\nabla (T_{k}(u_{n})-T_{k}(u))\nonumber\\
&\qquad+\nu\int_{\omega}
u_{n}^{r(x)}(T_{k}(u_{n})-T_{k}(u))\nonumber\\
&=\beta\int_{\Omega}\left(\vert \nabla T_{k}(u_{n})\vert^{p(x)-2}\nabla T_{k}(u_{n})-\vert \nabla T_{k}(u)\vert^{p(x)-2}\nabla T_{k}(u)\right).\nabla (T_{k}(u_{n})-T_{k}(u))\phi\nonumber\\
&\qquad+\nu\int_{\Omega}
u_{n}^{r(x)}(T_{k}(u_{n})-T_{k}(u))\phi\nonumber\\
&=\beta\int_{\Omega}\vert \nabla T_{k}(u_{n})\nabla T_{k}(u_{n})\vert^{p(x)-2}\nabla[(T_{k}(u_{n})-T_{k}(u))\phi]\nonumber\\
&\qquad-\beta\int_{\Omega}\vert \nabla T_{k}(u_{n})\vert^{p(x)-2}\nabla T_{k}(u_{n}).\nabla\phi[T_{k}(u_{n})-T_{k}(u)]\nonumber\\
&\qquad-\beta\int_{\Omega}\vert \nabla T_{k}(u)\vert^{p(x)-2}\nabla T_{k}(u).\nabla(T_{k}(u_{n})-T_{k}(u))\phi\nonumber\\
&\qquad+\nu\int_{\Omega}
u_{n}^{r(x)}(T_{k}(u_{n})-T_{k}(u))\phi
\end{align}
 By Lemma \ref{lem5}, we obtain
$$ u_{n}^{r(x)}\rightarrow u^{r(x)}\quad \mbox{strongly in }\; L^{1}(\Omega).$$
Therefore, since $T_{k}(u_{n})$ strongly converges to $T_{k}(u)$ in $L^{p(.)}(\Omega)$ (Lemma \ref{lem3}), we have
\begin{equation}
\label{39}
\int_{\Omega}
u_{n}^{r(x)}(T_{k}(u_{n})-T_{k}(u))\phi\rightarrow0,\quad \mbox{as}\;\; n\rightarrow\infty.
\end{equation}
We can easily know the fact that $\vert \nabla T_{k}(u)\vert^{p(x)-2}\nabla T_{k}(u)\in L_{loc}^{p'(.)}(\Omega)$, and $\nabla(T_{k}(u_{n})-T_{k}(u))\phi$ tends to zero weakly in $L^{p}(\Omega)$, we get
\begin{equation}
\label{40}
\int_{\Omega}\vert \nabla T_{k}(u)\vert^{p(x)-2}\nabla T_{k}(u).\nabla(T_{k}(u_{n})-T_{k}(u))\phi\rightarrow0,\quad \mbox{as}\;\; n\rightarrow\infty.
\end{equation}
$\nabla\phi[T_{k}(u_{n})-T_{k}(u)]$ strongly converges to zero in $L^{p(.)}(\Omega)$. Thus
\begin{equation}
\label{41}
\int_{\Omega}\vert \nabla T_{k}(u_{n})\vert^{p(x)-2}\nabla T_{k}(u_{n}).\nabla\phi[T_{k}(u_{n})-T_{k}(u)]\rightarrow0,\quad \mbox{as}\;\; n\rightarrow\infty.
\end{equation}
From \eqref{38}-\eqref{41}, we have
$$\int_{\omega}\left(\vert \nabla T_{k}(u_{n})\vert^{p(x)-2}\nabla T_{k}(u_{n})-\vert \nabla T_{k}(u)\vert^{p(x)-2}\nabla T_{k}(u)\right).\nabla (T_{k}(u_{n})-T_{k}(u))\rightarrow0,$$
then $T_{k}(u_{n})$ strongly converges to $T_{k}(u)$ in $W_{0}^{1,p(.)}(\omega)$ for all $k>0$,
i.e., since $\omega$ is arbitrary, that $T_{k}(u_{n})$ strongly converges to $T_{k}(u)$ in$W_{loc}^{1,p(.)}(\Omega)$.
\par
 Choosing $\phi\equiv1$ and repeating the same proof, we obtain that $T_{k}(u_{n})$ strongly converges to $T_{k}(u)$ in $W_{0}^{1,p(.)}(\Omega)$, then $\nabla u_{n}\rightarrow \nabla u$ a.e. in $\Omega$.
\end{proof}
\begin{proof}[Proof of the Theorem \ref{thm0}]
It is easy to pass to the limit in the right hand side of problems
\eqref{Pn}. On the other hand, using Lemma \ref{lem2}, we have
$$0\leq \bigg\vert \frac{f_{n}\varphi}{\left(u_{n}+\frac{1}{n}\right)^{\gamma(x)}}\bigg\vert\leq \frac{\Vert \varphi\Vert_{\infty}}{c_{\omega}^{\gamma^{-}}}f,$$
for every $\varphi\in C_{0}^{1}(\Omega)$, using Lebesgue Theorem and \eqref{00},  we conclude that
\begin{equation}
\label{42}
\lim\limits_{n\rightarrow\infty}
\int_{\Omega}\frac{f_{n}\varphi}{\left(u_{n}+\frac{1}{n}\right)^{\gamma(x)}}=\int_{\Omega}
\frac{f\varphi}{u^{\gamma(x)}}.
\end{equation}
By the same argument, we get
\begin{equation}
\label{43}
\lim\limits_{n\rightarrow\infty}
\int_{\Omega}b(x) u_{n}^{r(x)}\varphi=\int_{\Omega}
u^{r(x)}\varphi.
\end{equation}
For the first term, by Proposition \ref{prop0} we have that
$$a(x)\vert \nabla u_{n}\vert^{p(x)-2}\nabla u_{n} \rightarrow a(x)\vert \nabla u\vert^{p(x)-2}\nabla u\quad \mbox{a.e. in}\; \Omega,$$
furthermore $a(x)\vert \nabla u_{n}\vert^{p(x)-2}\nabla u_{n}$ is majorette by $\beta \vert \nabla u_{n}\vert^{p(x)-1}$. Observe that $p(x)-1<q(x)$, by Lemma \ref{lem6} and Vitali's Theorem, we have
\begin{equation}
\label{44}
\lim\limits_{n\rightarrow\infty}
a(x)\vert \nabla u_{n}\vert^{p(x)-2}\nabla u_{n}.\nabla\varphi=\int_{\Omega}
a(x)\vert \nabla u\vert^{p(x)-2}\nabla u.\nabla\varphi.
\end{equation}
Hence from \eqref{42}-\eqref{43} we can deduce \eqref{6}.
\end{proof}


\end{document}